\documentclass[12pt]{amsart}

\usepackage[utf8]{inputenc}
\usepackage[T1]{fontenc}

\usepackage{verbatim} 
\usepackage[matrix, arrow, cmtip]{xy}
\usepackage{tikz}
\usetikzlibrary{matrix, positioning, arrows.meta}

\usepackage{amssymb, enumerate, enumitem, longtable, float, graphicx, array}

\makeatletter
\@namedef{subjclassname@2010}{%
  \textup{2010} Mathematics Subject Classification}
\makeatother

\newtheorem{theorem}{Theorem}[section]
\newtheorem{lemma}[theorem]{Lemma}

\newtheorem{conjecture}[theorem]{Conjecture}
\newtheorem*{carrier theorem}{Carrier Theorem}
\theoremstyle{definition}
\newtheorem*{definition}{Definition}

\newtheorem{construction}[theorem]{Construction}
\theoremstyle{remark}

\newcommand{\B}[1]{{\mathcal{B}_{#1}}}
\newcommand{\F}[1]{{\mathcal{F}_{#1}}}

\newcommand{\G}[1]{{\mathcal{G}_{#1}}}
\newcommand{\U}[1]{{\mathcal{U}_{#1}}}
\newcommand{\V}[1]{{\mathcal{V}_{#1}}}
\newcommand{\W}[1]{{\mathcal{W}_{#1}}}
\newcommand{\N}{{\mathcal{N}}}
\newcommand{\Akn}{{\mathcal{A}_{\kappa,n}}}
\renewcommand{\O}[1]{{\mathcal{O}_{#1}}}
\newcommand{\df}[1]{{\bf #1}}
\newcommand{\bts}{{{\sc(bts}$_{\kappa,n}${\sc)}}}
\newcommand{\qts}{{{\sc(qts}$_{\kappa,n}${\sc)}}}
\newcommand{\nuk}{{\nu^n_\kappa(K)}}

\DeclareMathOperator{\diam}{diam}
\DeclareMathOperator{\st}{st}

\DeclareMathOperator{\ost}{ost}

\DeclareMathOperator{\mesh}{mesh}
\DeclareMathOperator{\Cl}{Cl}

\begin{document}

\baselineskip=17pt

\title{A construction of {N}\"obeling manifolds of arbitrary weight}

\author{G.~C.~Bell}
\address[G.~C.~Bell]{Department of Mathematics and Statistics, University of North Carolina at Greensboro, Greensboro, NC 27412, USA}
\email{gcbell@uncg.edu}

\author[A.~Nagórko]{A.~Nagórko}
\address[A.~Nagórko]{Faculty of Mathematics, Informatics, and Mechanics, University of Warsaw, Banacha 2, 02-097 Warszawa, Poland}
\email{amn@mimuw.edu.pl}
\thanks{This research was supported by the NCN (Narodowe Centrum Nauki) grant no. 2011/01/D/ST1/04144.}

\date{}

\begin{abstract}
For each cardinal $\kappa$, each natural number $n$ and each simplicial complex $K$ we construct a space $\nuk$ and a map $\pi \colon \nuk \to K$ such that the following conditions are satisfied.
\begin{enumerate}
\item $\nuk$ is a complete metric $n$-dimensional space of weight $\kappa$.
\item $\nuk$ is an absolute neighborhood extensor in dimension $n$.
\item $\nuk$ is strongly universal in the class of $n$-dimensional complete metric spaces of weight~$\kappa$.
\item $\pi$ is an $n$-homotopy equivalence.
\end{enumerate}
For $\kappa = \omega$ the constructed spaces are $n$-dimensional separable N\"obeling manifolds.
The constructed spaces have very interesting fractal-like internal structure
  that allows for easy construction, subdivision, and surgery of brick partitions.
\end{abstract}

\subjclass[2010]{Primary 55M10, 54F45; Secondary 54C55}

\keywords{N\"obeling space}

\maketitle

\section{Introduction}

An $n$-dimensional N\"obeling space is a subset $\nu^n$ of $\mathbb{R}^{2n+1}$ consisting of points with at most $n$ rational coordinates.
They were constructed in~$1931$ by N\"obeling~\cite{nobeling1931}. The $n$-dimensional N\"obeling space is a universal space in the class of $n$-dimensional separable metric spaces; i.e., it contains a topological copy of every separable metric $n$-dimensional space as a subspace.

In the 1980s it was conjectured that $\nu^n$ is an $n$-dimensional analog of the Hilbert space~$\ell^2$. 
The conjecture was inspired by the observation that $\nu^n$ satisfies $n$-dimensional analogs of properties that topologically characterize the Hilbert space.
The analogy is formalized in the following theorem.
\begin{theorem}
A space $X$ is homeomorphic to $\nu^n$ if and only if the following conditions are satisfied.
\begin{enumerate}
\item $X$ is a separable complete metric $n$-dimensional space.
\item $X$ is an absolute extensor in dimension $n$.
\item $X$ is strongly universal in the class of $n$-dimensional separable metric metric spaces (every continuous map from a separable $n$-dimensional metric space into $X$ can be arbitrarily closely approximated by closed embeddings).
\end{enumerate}
\end{theorem}
Observe that for $n=\infty$ we have $\nu^\infty = \mathbb{R}^\infty$, which is homeomorphic to $\ell^2$. 
Hence for $n = \infty$ this theorem is a reformulation of a famous characterization theorem of Toruńczyk~\cite{torunczyk1981}.
For $n < \infty$ this was a long-standing open conjecture that was proved independently in $2006$ in~\cite{ageev2007a,ageev2007b,ageev2007c,levin2006,nagorkophd,nagorko2013}.

Toruńczyk gave topological characterization of non-separable Banach spaces of weight $\kappa$ that concluded a program of topological characterization of Banach spaces that was started by Fr\'echet \cite{torunczyk1981}.
The analoguous result for N\"obeling spaces is not known.
We state it in the form of the following conjecture, which is also a rigidity
  theorem for N\"obeling manifolds. 
A space is a N\"obeling manifold if it is locally homeomorphic to a N\"obeling space.
There is analogous rigidity result for separable N\"obeling manifolds proved in~\cite{nagorko2013}.
In particular it implies that a separable N\"obeling manifold is homeomorphic to $\nu^n$ if and only if it has vanishing homotopy groups in dimensions less than $n$.
\begin{definition}
  An \df{abstract $n$-dimensional N\"obeling manifold of weight $\kappa$} is a space that satsifies the following conditions.
  \begin{enumerate}
\item $X$ is an $n$-dimensional complete metric space of weight $\kappa$.
\item $X$ is an absolute neighborhood extensor in dimension $n$.
\item $X$ is strongly universal in the class of $n$-dimensional complete metric spaces of weight~$\kappa$.
\end{enumerate}
  An \df{abstract $n$-dimensional N\"obeling space of weight $\kappa$} is an abstract $n$-dimensional N\"obeling manifold of weight $\kappa$ that has vanishing homotopy groups of dimensions less than~$n$.
\end{definition}
\begin{conjecture}
  Two abstract $n$-dimensional N\"obeling manifolds of weight $\kappa$ are homeomorphic if and only if they are $n$-homotopy equivalent.
\end{conjecture}

Note that in the separable case, by Open Embedding Theorem~\cite{nagorko2013}, every separable $n$-dimensional N\"obeling manifold is homeomorphic to an open subset of $\nu^n$.

In the present paper we construct abstract N\"obeling spaces of arbitrary weight $\kappa$ and abstract N\"obeling manifolds that are $n$-homotopy equivalent to an arbitrary simplicial complex $K$.
\begin{theorem}
  For each cardinal $\kappa$ and each simplicial complex $K$ there exists an abstract $n$-dimensional N\"obeling manifold $\nuk$ and a map $\pi \colon \nuk \to K$ that is an $n$-homotopy equivalence.
\end{theorem}

It is acknowledged that the difficulty of proving characterization theorem for separable $n$-dimensional N\"obeling space $\nu^n$ for $1 \leq n < \infty$ lay in the fact that both $\nu^0$ and $\nu^\infty$ possess a natural product structure, while $\nu^n$ does not~\cite{chigogidze1996}.
The spaces that we construct have very nice internal fractal-like structure, 
  as they are constructed as inverse limits of sequences whose bonding maps possess a high degree of symmetry.
This allows for easy construction of brick partitions~\cite{kawamura1995} for these spaces, easy subdivision of these partitions, and allows for easy surgery on these brick partitions.
The construction is new and interesting even in the separable case ($\kappa = \omega$), where the characterization theorem is known.

Note that the constructed spaces in the separable case are similar to the universal space $\mathbb{U}_n$ constructed in~\cite{bellnagorko2013}, although the present construction is much simpler. Similar arguments to the ones given here show that $\mathbb{U}_n$ is homeomorphic to $\nu^n$, which answers a question stated in~\cite{bellnagorko2013}.

The spaces $\nuk$ that we construct are Markov spaces in the sense of~\cite{bellnagorko2017-markov}. Theorem~\ref{thm:awn universal} gives a sufficient condition for a Markov space to be strongly universal for the class of $n$-dimensional metric spaces of weight~$\kappa$.

\section{Preliminaries}

In this section we set the basic definitions and reference known results that will be used in the later sections.

\subsection{Absolute extensors in dimension $n$}

\begin{definition}
  We say that a space $X$ is \df{$k$-connected} if each map $\varphi \colon S^k \to X$ from a $k$-dimensional sphere into $X$ is null-homotopic in $X$.
  We let $\mathcal{C}^{n-1}$ denote the class of all spaces that are $k$-connected for each $k < n$.
\end{definition}

\begin{definition}
  We say that a space $X$ is \df{locally $k$-connected} if for each point $x \in X$ and each open neighborhood $U \subset X$ of $x$ there exists an open neighborhood $V$ of $x$ such that each map $\varphi \colon S^k \to V$ from a $k$-dimensional sphere into $V$ is null-homotopic in $U$.
  We let $\mathcal{LC}^{n-1}$ denote the class of all spaces that are locally $k$-connected for each $k < n$.
\end{definition}

\begin{definition}
  We say that a metric space $X$ is an \df{absolute neighborhood
    extensor in dimension~$n$} if every map into $X$
  from a closed subset $A$ of an $n$-dimensional metric space extends
  over an open neighborhood of $A$. The class of absolute neighborhood
  extensors in dimension $n$ is denoted by $ANE(n)$ and its elements
  are called $ANE(n)$-spaces.
\end{definition}

\begin{definition}
We say that a metric space $X$ is an \df{absolute extensor 
  in dimension~$n$} if every map into $X$ from a closed subset 
  of an $n$-dimensional metric space $Y$ extends over the 
  entire space $Y$.
The class of absolute extensors in dimension~$n$ is denoted by
  $AE(n)$ and its elements are called $AE(n)$-spaces.
\end{definition}

Absolute extensors and absolute neighborhood extensors in dimension $n$ were characterized by Dugundji in the following theorem.

\begin{theorem}[\cite{dugundji1958}]\label{thm:dugundji}
  Let $X$ be a metric space. The following conditions are equivalent. 
	\begin{itemize}
		\item $X\in ANE(n)\iff X$ is locally $k$-connected for all $k\le n$, i.e. $X\in\mathcal{LC}^{n-1}$.  
		\item $X\in AE(n)\iff X\in ANE(n)$ and if $X$ is $k$-connected for all $k\le n$, i.e., $X\in \mathcal{C}^{n-1}$.
	\end{itemize}
\end{theorem}

\subsection{Simplicial complexes}

For the reasons given in~\cite{bellnagorko2017-localk},
{\bf we always assume the metric topology on simplicial complexes}.

\begin{lemma}\label{lem:complex is ane}
  A locally finite-dimensional simplicial complex is a complete metric $ANE(\infty)$-space.
\end{lemma}

\subsection{$n$-Homotopy equivalence}

\begin{definition}
  We say that a map is a \df{weak $n$-homotopy equivalence} if it induces isomorphisms on homotopy groups of dimensions less than $n$, regardless of the choice of base point.
\end{definition}

\begin{definition}
  We say that maps $f, g \colon X \to Y$ are \df{$n$-homotopic} if for every map $\varPhi$ from a complex of dimension less than $n$ into $X$, the compositions $f \circ \varPhi$ and $g \circ \varPhi$ are homotopic in the usual sense.
\end{definition}

\begin{theorem}[\cite{nagorko2013}]
  A map of two $n$-dimensional $ANE(n)$-spaces is a weak $n$-homotopy equivalence if and 
  only if it is an $n$-homotopy equivalence.
\end{theorem}

\subsection{Carrier Theorem}

The Carrier Theorem is proved in~\cite{nagorko2007}.

\begin{definition}
  Let~$\mathcal{C}$ be a class of topological spaces. A locally finite locally
  finite-dimensional closed $AE(\mathcal{C})$-cover is said to be
  \df{regular for the class~$\protect\mathcal{C}$}. Recall that a cover is 
  locally finite-dimensional if its nerve is locally finite-dimensional.
\end{definition}

\begin{definition}
  A \df{carrier} is a function $C \colon \F{} \to \G{}$ from
  a cover~$\F{}$ of a space~$X$ into a collection~$\G{}$ of subsets of
  a topological space such that for each $\mathcal{A} \subset \F{}$ if
  $\bigcap \mathcal{A} \neq \emptyset$, then $\bigcap_{A \in
  \mathcal{A}} C(A) \neq \emptyset$. We say that a map~$f$ is
  \df{carried by~$C$} if it is defined on a
  closed subset of~$X$ and $f(F) \subset C(F)$ for each $F \in \F{}$.
\end{definition}

\begin{carrier theorem}
  Assume that $C \colon \F{} \to \G{}$ is a carrier such that~$\F{}$ is a
  closed cover of a space~$X$ and~$\G{}$ is an $AE(X)$-cover of another space.
  If~$\F{}$ is locally finite and locally finite-dimensional, then each map
  carried by~$C$ extends to a map of the entire space~$X$, also carried
  by~$C$.
\end{carrier theorem}

\subsection{Local $k$-connectedness of inverse limits}

The following theorem is proved in~\cite{bellnagorko2017-localk}. We cite it here for completeness.

\begin{definition}
  Let $K$ and $L$ be simplicial complexes.
  We say that a map $p \colon K \to L$ is \df{$n$-regular} if it is quasi-simplicial (i.e. it is simplicial into the first barycentric subdivision $\beta K$ of $K$) and if for each simplex $\delta$ of $\beta K$ the inverse image $p^{-1}(\delta)$ has vanishing homotopy groups in dimensions less than $n$ (regardless of the choice of a base point).
\end{definition}

\begin{theorem}\label{thm:ane limit of polyhedra}
  Let
  \[
    X = \lim_{\longleftarrow} K_1 \xleftarrow{p_1} K_2 \xleftarrow{p_2} \cdots.
  \]
  Assume that for each $i$ the following conditions are satisfied:
  \begin{enumerate}[label=(\Roman*)]
  \item $K_i$ is a simplicial complex with the metric topology; and
  \item $p_i$ is surjective and $n$-regular.
  \end{enumerate}
  Then 
  \begin{enumerate}
  \item $X$ is an $ANE(n)$;
  \item each short projection $\pi^k_i \colon K_k \to K_i$ and each long projection $\pi_i \colon X \to X_i$ is a weak $n$-homotopy equivalence;
  \item for each $i$, the covers $\O{i}$ and $\B{i}$ are $AE(n)$-covers of $X$.
  \end{enumerate}  
\end{theorem}

See Section~\ref{sec:regular covers} for the definition of $\O{i}$. We do not use covers $\B{i}$ in the present paper.

\section{A lifting property}

In this section we define a lifting property \qts\ of quasi-simplicial maps. An assumption that bonding maps of an inverse sequence satisfy \qts\ is sufficient for strong universality of its inverse limit.  However, we immediately show that if we work with simplicial complexes of weight $\kappa$, then the map that satisfies \qts is unique up to a simplicial isomorphism.

\begin{definition}
  Let $p \colon K \to L$ be a quasi-simplicial map.
  We say that~$p$ has the \df{quasi-simplicial-to-simplicial
    lifting property} with respect to $n$-dimensional
    complexes of weight $\kappa$ if the following condition is satisfied.

\vspace{2mm}
\begin{tabular}{ll}
  \qts & 
\begin{minipage}{290pt}
  \vspace{1mm}
  \em
  For each pair $A \supset B$ of at most $n$-dimensional simplicial complexes of weight at most $\kappa$ and each commutative diagram
  \[
    \xymatrix@M=8pt@C=35pt{
    K \ar[r]^{p} & L  \\
    B \ar[r]^{\subset} \ar[u]_{g} & A \ar[u]^{G} \ar@{..>}[ul]^{\tilde g} \\
    }
  \]
  such that $g$ is a simplicial map and $G$ is a quasi-simplicial map, there exists a simplicial map $\tilde g \colon A \to K$ such that $g$ is an embedding on $B \setminus A$, $p \circ \tilde g = G$ and $\tilde g$ is equal to $g$ on $B$.
  \em
  \vspace{1mm}
\end{minipage}
\end{tabular}
\end{definition}

\begin{lemma}\label{lem:qts fibers}
  If $p \colon K \to L$ is a quasi-simplicial map that satisfies \qts, then for each simplex $\delta \in \beta L$ the inverse image $p^{-1}(\delta)$ is a full $n$-simplex with at least $\kappa$ vertices. In particular, $p$ is $n$-regular.
\end{lemma}
\begin{proof}
  let $D$ denote a simplicial complex with $\kappa$ vertices and no higher dimensional simplices.
  Let $v$ be a vertex of $\beta L$.
  Let $f \colon D \to \{ v \}$ be a constant map.
  Since $f$ is simplicial, by \qts it lifts to a simplicial embedding into $p^{-1}(v)$, hence $p^{-1}(v)$ has at least $\kappa$ vertices.
  The same holds for $p^{-1}(\delta)$, where $\delta$ is any simplex in $\beta L$.
  
  Fix $\delta$ in the triangulation $\tau \beta L$. 
  Let $A$ be a set of vertices in $p^{-1}(\delta)$ with $\# A \leq n+1$.
  Let $\Delta$ be a simplex spanned by $A$.
  Let $G \colon \Delta \to L$ be a simplicial map such that $G_{|A} = p_{|A}$.
  Such a $G$ exists, since $p(A) \subset \delta$.
  The dimension of $\Delta$ is at most $n$, hence by \qts, $G$ lifts to a map $\tilde g$ into $K$ such that $\tilde g_{|A} = id_A$.
  Hence vertices of $A$ span a simplex in $K$, therefore $p^{-1}(\delta)$ is a full $n$-simplex with at least $\kappa$ vertices.
  
  Since every full $n$-simplex is $AE(n)$, $p$ is $n$-regular.
\end{proof}

\begin{lemma}\label{lem:qts} 
  Let $p_1 \colon K_1 \to L$ and $p_2 \colon K_2 \to L$ be a quasi-simplicial maps that satisfy \qts.
  If $K_1$ and $K_2$ have weight at most $\kappa$, then there exists a simplicial isomorphism $h \colon K_1 \to K_2$ such that $p_2 \circ h = p_1$.
\end{lemma}
\begin{proof}
  Since the cardinality of the vertex sets of $K_1$ and $K_2$ is bounded by $\kappa$, Lemma~\ref{lem:qts fibers} implies that for each simplex $\delta$ in the triangulation of $\beta L$, the inverse images $p_1^{-1}(\delta)$ and $p_2^{-1}(\delta)$ are full $n$-simplexes.
  Hence any bijection $h_0$ that maps vertices of $K_1$ onto vertices of $K_2$ and with the property that $p_2 \circ h_0 = p_1$ extends to a simplicial isomorphism.
\end{proof}

\section{Regular covers of inverse limits of polyhedra}\label{sec:regular covers}

\begin{definition}
 Let
  \[
    X = \lim_{\longleftarrow} K_1 \xleftarrow{p_1} K_2 \xleftarrow{p_2} \cdots.
  \]
  We let $\pi^k_i \colon K_k \to K_i$ denote the short projections and $\pi_i \colon X \to X_i$ denote the long projections.
\end{definition}

\begin{definition}
  Let $K$ be a simplicial complex. Let $L \subset K$ be a subcomplex of $K$.
  The \df{open star $\ost_K L$ of $L$ in $K$} is the complement of the union of all simplices of $K$ that do not intersect $L$:
  \[
    \ost_K L = K \setminus \bigcup \{ \delta \in \tau(K) \colon \delta \cap L = \emptyset \}\text{.}
  \]
\end{definition}

\begin{definition}\label{def:oi}
  Let
  \[
  X = \lim_{\longleftarrow} \left(   
  K_1\xleftarrow{p_1}K_2\xleftarrow{p_2}\cdots \right)
  \]
  Let $v(K_i)$ denote the set of vertices of $K_i$.
  We let
  \[
    \O{K_i} = \{ O_v = \ost_{K_i} v \}_{v \in v(K_i)}
  \]
  be the cover of $K_i$ by the open stars of vertices of $K_i$ and
  \[
    \O{i} = \{ \pi^{-1}_i(\ost_{K_i} v) \}_{v \in v(K_i)}
  \]
  be the cover of $X$ by the sets of threads that pass through elements of $\O{K_i}$.
\end{definition}

\begin{definition}
  We write $\F{} < \G{}$ if the following condition is satisfied:
  \begin{quote}
    for each $F_1, F_2 \in \F{}$ if $F_1 \cap F_2 \neq \emptyset$, then there exists $G \in \G{}$ such that $F_1 \cup F_2 \subset G$.
  \end{quote}
\end{definition}

\begin{definition}
  Let $K$ be a simplicial complex. Let $\kappa \in (0, \infty)$.
  We define \df{a geodesic metric of scale $\kappa$ on $K$} in the following way.
  On each $\delta \in \tau(K)$ we take a Euclidean metric with edge length $\kappa$. 
  We extend this to the unique geodesic metric on $K$.
\end{definition}

\begin{lemma}\label{lem:oi}
  Let
  \[
  X = \lim_{\longleftarrow} \left(   
  K_1\xleftarrow{p_1}K_2\xleftarrow{p_2}\cdots \right),
  \]
  where each $p_i$ is a surjective $n$-regular quasi-simplicial map.
  If we endow $K_i$ with a geodesic metric of scale $2^{-i/2}$, then
  \begin{enumerate}
    \item $\sum_i \mesh \O{i} < \infty$;
    \item $\O{i}$ is an open $AE(n)$-cover of $X$;
    \item $\O{i+1} < \O{i}$
  \end{enumerate}
\end{lemma}
\begin{proof}
  Condition (1) follows from the choice of metrics on the $K_i$.
  
  Condition (2) follows from Theorem~\ref{thm:ane limit of polyhedra}.
  
  To show condition (3), observe that if $\ost v \cap \ost w \neq \emptyset$, where $v, w$ are vertices in $K_{i+1}$, then both $p_i(v)$ and $p_i(w)$ are adjacent to a single vertex $z$ of $K_i$ (since $p_i$ is quasi-simplicial). Then $\ost v \cup \ost w \subset \ost z$.
\end{proof}

\begin{lemma}\label{lem:small steps}
  Let $\U{i}$ be a sequence of covers with the property that $\U{i} < \U{i+1}$.
  Let $x_i$ be a sequence of points such that $x_{i + 1} \in \st_{\U{i}} x_i$.
  Then for each $k$ and each $i \geq k$ we have $x_i \in \st^2_{\U{k}} x_k$.
\end{lemma}
\begin{proof}
  Fix $i \geq k$. 
  Since $x_i \in \st_\U{i-1} x_{i-1}$, there exists an element $U_{i-1} \in \U{i-1}$ such that $x_i, x_{i-1} \in U_{i-1}$.
  Likewise, there exists $U_{i-2} \in \U{i-2}$ such that $x_{i-1}, x_{i-2} \in U_{i-2}$.
  Since $\U{i-1} \prec \U{i-2}$ there exists $U'_{i-2}$ such that $U_{i-1} \subset U'_{i-2}$.
  Since $\U{i-2} < \U{i-3}$, there exists $U'_{i-3}$ such that $U'_{i-2} \cup U_{i-2} \subset U'_{i-3}$. 
  Therefore $x_i \in \st^2_\U{i-3} x_{i-3}$.
  A recursive application of the above argument finishes the proof.
\end{proof}

\section{Strong $\Akn$-universality of inverse limits of polyhedra}

\begin{lemma}\label{lem:strong carrier}
  Let
  \[
  X = \lim_{\longleftarrow} \left(   
  K_1\xleftarrow{p_1}K_2\xleftarrow{p_2}\cdots \right).
  \]
  
  Let $\pi_i \colon X \to K_i$ denote the long projection as above.
  Let $Y$ be a metric space of dimension at most $n$.
  Let $\W{}$ be an open locally finite cover of $Y$ with multiplicity at most $n + 1$.
  Let $C \colon \W{} \to \O{i}$ be a carrier.
  Let $v_W$ denote a vertex of $K_i$ such that $C(W) = \ost v_W$.
  For $y \in Y$ let $\Delta_y$ denote a simplex in $K_i$ that is spanned by $\{ v_W \colon W \in \W{}, x \in W \}$.
  
  Then there exists a map $f \colon Y \to X$ such that for each $y \in Y$ we have
  \[
    g(y) \in \pi_i^{-1}(\Delta_y).
  \]
  
  In particular, if $C$ is one-to-one, then
  \[
    \mesh g^{-1}(\O{i}) \leq \mesh \st \W{}.
  \]
\end{lemma}
\begin{proof}
  Let
  \[
    Y_k = \{ y \in Y \colon \# \{ W \in \W{} \colon y \in W \} \leq k \}
  \]
  be a set of points in $Y$ with $\W{}$-multiplicity at most $k$.
  Because $\W{}$ is open, each $Y_k$ is closed.
  We will define a sequence $g_k \colon Y_k \to X$, $k = 0, 1, 2, 3, \ldots, n+1$ such that for $k \geq 1$:
  \begin{enumerate}
  \item $g_{k | Y_{k-1}} = g_{k-1}$; and
  \item $g_k(y) \in \pi^{-1}_i(\Delta_y)$.
  \end{enumerate}
  The map $g_{n+1}$ will be the map $g$ that we are looking for.
  
  We let $g_0$ to be an empty map. Let $k > 0$ and assume that we have already constructed $g_{k-1}$.
  Let
  \[
    \V{}_k = \{ \cap \mathcal{A} \colon \mathcal{A} \subset \W{}, \# \mathcal{A} = k \}.
  \]
  Let $C_k \colon \V{}_k \to X$ be defined by the formula
  \[
    C_k(W_1 \cap W_2 \cap \ldots \cap W_k) = \pi_i^{-1}(\Delta_{v_{W_1}, v_{W_2}, \ldots, v_{W_k}}),
  \]
  where $\Delta_{v_{W_1}, v_{W_2}, \ldots, v_{W_k}}$ denotes the simplex spanned in $K_i$ by vertices $v_{W_1}, v_{W_2}, \ldots, v_{W_k}$.
  
  Observe that for each $k$, $C_k$ is a carrier. 
  We construct $g_k$ in such a way that additionally the following condition is satisfied:
  \begin{enumerate}
  \item[(3)] $g_{k | Y_k \setminus Y_{k-1}}$ is carried by $C_k$.
  \end{enumerate}
  Note that $g_0$ is trivially carried by $C_0$.
  
  Since $g_{k-1}$ is carried by $C_{k-1}$ (as a map from $Y_{k-1}$), it is also carried by $C_k$. Therefore, by the Carrier Theorem, we can extend $g_{k-1}$ to a map $g_k \colon Y_k \to X$ that is carried by $C_k$. This map satisfies conditions (1), (2) and (3) and we are done.
\end{proof}

\begin{lemma}\label{lem:lift}
  Let
  \[
  X = \lim_{\longleftarrow} \left(   
  K_1\xleftarrow{p_1}K_2\xleftarrow{p_2}\cdots \right).
  \]
  Assume that each bonding map $p_i$ satisfies
  \qts.
  Let $Y$ be a metric space of dimension at most $n$
  and weight at most~$\kappa$.
  Let $A$ be a closed subset of $Y$.
    
  For each map $f \colon Y \to X$, each $\varepsilon > 0$,
  each open neighborhood $U$ of $A$ in $Y$ and and each $i$ there exists a map $g \colon Y \to X$ that satisfies the following conditions:
  \begin{enumerate}
  \item $g_{|A} = f_{|A}$;
  \item $g$ is $\O{i}$-close to $f$; and
  \item $\mesh g^{-1}_{|Y\setminus U}(\O{i+1}) \leq \varepsilon$.
  \end{enumerate}
\end{lemma}
\begin{proof}
  Let $\W{}$ be a cover of $Y$ that satisfies the following conditions:
  \begin{enumerate}
  \item $\W{} \prec f^{-1}(\O{i+1})$;
  \item $\W{}$ is open, has multiplicity at most $n+1$,
    is locally finite and has cardinality at most $\kappa$;
  \item $\mesh \st \W{} \leq \varepsilon$; and
  \item $\st_\W{} A \subset U$.
\end{enumerate}


  Let $C \colon f^{-1}(\O{i}) \to \O{i}$ be defined by the formula $C(f^{-1}(O)) = O$ for each $O\in \O{i}$.
  It is a carrier and $f$ is carried by $C$.
  Let $\tilde C \colon N(f^{-1}(\O{i})) \to N(\O{i}) = K_i$ be the map between nerves that is induced by $C$.

  Let $\mathcal{R} = \{ F \in \W{} \colon F \cap A \neq \emptyset \}$. Let $J \colon N(\mathcal{R}) \to N(\W{})$
  be the identity map ($N(\mathcal{R})$ is a subcomplex of $N(\W{})$).

  Let $D \colon f^{-1}(\O{i+1}) \to \O{i+1}$ be defined by the formula $D(f^{-1}(O)) = O$ for each $O\in \O{i+1}$.

  Let $K \colon N(\W{}) \to N(f^{-1}(\O{i+1}))$ be a map such that if $K(v(F)) = v(U)$, then $F \subset U$. Such a map exists because $\W{}$ refines $f^{-1}(\O{i+1})$; moreover, it is a simplicial map.

  Let $L \colon N(f^{-1}(\O{i+1})) \to N(f^{-1}(\O{i}))$ be a map induced by inclusions. The following diagram is commutative.

\begin{center}
\begin{tikzpicture}
  \matrix (m) [matrix of math nodes,row sep=3em,column sep=2.5em,minimum width=2em]
  {
     & & N(f^{-1}(\O{i})) & N(\O{i}) = K_i \\
     N(\mathcal{R}) & N(\W{}) & N(f^{-1}(\O{i+1})) & N(\O{i+1}) = K_{i+1} \\};
  \path[-stealth]
    (m-2-1) edge node [above] {$J$} (m-2-2)
    (m-2-2) edge node [above] {$K$} (m-2-3)
    (m-2-3) edge node [left] {$L$} (m-1-3)
    (m-1-3) edge node [above] {$\tilde C$} (m-1-4)
    (m-2-3) edge node [above] {$\tilde D$} (m-2-4)
    (m-2-4) edge node [left] {$p_i$} (m-1-4)
;
\end{tikzpicture}  
\end{center}

It follows that the following diagram is commutative.
\begin{center}
\begin{tikzpicture}
  \matrix (m) [matrix of math nodes,row sep=3em,column sep=7em,minimum width=2em]
  {
     N(\W{}) & N(\O{i}) = K_i \\
     N(\mathcal{R}) & N(\O{i+1}) = K_{i+1} \\};
  \path[-stealth]
  (m-2-1) edge node [above] {$\tilde D \circ K \circ J$} (m-2-2)
  (m-2-1) edge node [left] {$J$} (m-1-1)
  (m-2-2) edge node [left] {$p_i$} (m-1-2)
  (m-1-1) edge node [above] {$\tilde C \circ L \circ K$} (m-1-2)
  (m-1-1) edge[dotted] node [above] {$\tilde E$} (m-2-2)  
;
\end{tikzpicture}  
\end{center}

Since $p_i$ satisfies \qts{} we can lift $\tilde C \circ L \circ K$ to a simplicial map $\tilde E \colon N(\W{}) \to K_{i+1}$
such that 
\begin{enumerate}
\item $p_i \circ \tilde E = \tilde C \circ L \circ K$.
\item $\tilde E \circ J = \tilde D \circ K \circ J$.
\item $\tilde E$ is an embedding on $N(\W{}) \setminus N(\mathcal{R})$.
\end{enumerate}
Let $E \colon \W{} \to \O{i+1}$ be a map such that $\tilde E(v(F)) = v(E(F))$.
It is a carrier and $f_{|A}$ is carried by $E$.

Let $Z = Y \setminus \bigcup \mathcal{R}$.
By Lemma~\ref{lem:strong carrier} there exists a map $h \colon Z \to X$
  that is carried by $E_{|\W{} \setminus \mathcal{R}}$ and such that $\mesh h^{-1}(\O{i+1}) \leq \mesh \st \W{} \leq \varepsilon$.
  It is $\O{i}$-close to $f$ as both maps are carried by $C$.
  By Carrier Theorem, there exists a map $g \colon Y \to X$ that extends $f_{|A} \cup h$ and is carried by $C$.
  It satisfies conditions that we were looking for.
\end{proof}

\begin{lemma}\label{lem:closed embedding}
  Let $X$ be a Polish space. Let $f \colon X \to Y$ be a continuous map. 
  If for each $y \in Y$\[ \lim_{n \to \infty} \diam f^{-1}(B(y, \frac 1n)) = 0, \] then $f$ is a closed embedding of $X$ into $Y$.  
\end{lemma}
\begin{proof}
  If there are $x_0\neq x_1$ with $f(x_0) = f(x_1) = y$, then $\diam f^{-1}(B(y, \delta))>0$ contradicting the assumption. Hence $f$ is one-to-one. 
  We'll prove that $f$ is a closed map.
  Let $A$ be a closed subset of $X$. 
  Let $y_0  \in \Cl f(A)$.
  Let \[ A_n = f^{-1}\left(B\left(y_0, \frac 1n\right) \cap f\left(A\right)\right). \]
  By the assumption $\lim_{n \to \infty} \diam A_n = 0$. 
  Each $A_n$ is non-empty and closed and $X$ is complete.
  By Cantor's Intersection Theorem~\cite{engelking1989} there exists $x_0 \in \bigcap_{n \in \mathbb{N}} A_n$.
  We have $x_0 \in A$ and $f(x_0) = y_0$ by the definition of $A_n$.
  Therefore $y_0 \in f(A)$ and $f(A)$ is closed.
\end{proof}


\begin{theorem}\label{thm:awn universal}
  Let
  \[
  X = \lim_{\longleftarrow} \left(   
  K_1\xleftarrow{p_1}K_2\xleftarrow{p_2}\cdots \right).
  \]
  If each bonding map $p_i$ satisfies \qts, then $X$ is strongly $\Akn$-universal.
\end{theorem}
\begin{proof}
  Let $Y$ be a complete, at most $n$-dimensional metric space.
  Let $f \colon Y \to X$.
  Let $\U{}$ be an open cover of $X$.
  Let $\O{i}$ be a cover of $X$ as in Definition~\ref{def:oi}.
  Let
  \[
    R_i = \Cl_X \{ y \in Y \colon \forall_{U \in \U{}}  \st^2_\O{i} f(y) \not\subset U \}
  \]
  for $i > 1$ and let $R_1 = Y$.

  Let $g_1 = f$. By recursive application of Lemma~\ref{lem:lift} we construct
    a sequence of maps $g_i \colon X \to Y$ satisfying the following conditions.
  \begin{enumerate}
  \item $g_{i|R_i} = f_{|R_i}$
  \item $g_i$ is $\O{i-1}$-close to $g_{i-1}$
  \item $\mesh g_{i|Y\setminus R_i}^{-1}(\O{i-1}) \leq \frac 1i$.
  \end{enumerate}
  
  By (2) and by Lemma~\ref{lem:oi}, the sequence $g_i$ is uniformly convergent.
  As inverse limit of complete spaces, $X$ is complete.
  Therefore the limit
  \[
    g = \lim_{i \to \infty} g_i.
  \]
  exists and is continuous.
  Let $y \in Y$ and let $i$ such that $y \in Y \setminus R_i$ and $y \in R_{i-1}$.
  We have $g_{i-1}(y) = f(y)$.
  By (2) and by Lemma~\ref{lem:small steps}, $g(y) \in \Cl_X \st^2_\O{i} f(y)$.
  By the definition of $R_i$, there exists $U \in \U{}$ such that $f(y), g(y) \in U$. 
  Therefore $g$ is $\U{}$-close to $f$.
  
  Fix $y \in Y$ and $i \geq 1$. By Lemma~\ref{lem:small steps}, $f$ is $\st^2 \O{i}$-close to $g_i$.
  Therefore, $f^{-1}(\st_\O{i} y) \subset g_i^{-1}(\st_\O{i}^3 y)$ and $\diam g_i^{-1}(\st_\O{i}^3 y) \leq 5 \frac 1i$.

Therefore, $\lim_{n \to \infty} \diam f^{-1}(B(y, \frac 1n)) = 0$ and by Lemma~\ref{lem:closed embedding}, $f$ is a closed embedding.
\end{proof}

\section{Construction of N\"obeling manifolds of weight $\kappa$}

\begin{construction}
  Let $K$ be a simplicial complex.
  Let $\kappa$ be a cardinal number.
  Let $n$ be a natural number.
  We construct a simplicial complex $\N{}^n_\kappa(K)$ and a simplicial map $\pi \colon \N{}^n_\kappa(K) \to \beta K$.
  The complex $\N{}^n_\kappa(K)$ has $\kappa$ vertices corresponding to each vertex of $\beta K$.
  The map $\pi$ maps each vertex of $\N{}^n_\kappa(K)$ to a corresponding vertex of $\beta K$.
  The simplicial structure on $\N{}^n_\kappa(K)$ is the maximal $n$-dimensional simplicial structure such that $\pi$ is a simplicial map into $\beta K$. 
\end{construction}

\newcommand{\w}{22mm}
\begin{center}
\begin{longtable}{m{6mm}m{2mm}m{22mm}m{2mm}m{32mm}m{2mm}m{10mm}}
	\raisebox{4mm}{\includegraphics[height=4cm]{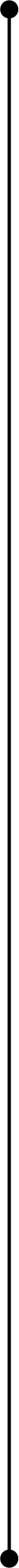}} &
	$\leftarrow$ &
	\raisebox{2mm}{\includegraphics[height=4.4cm]{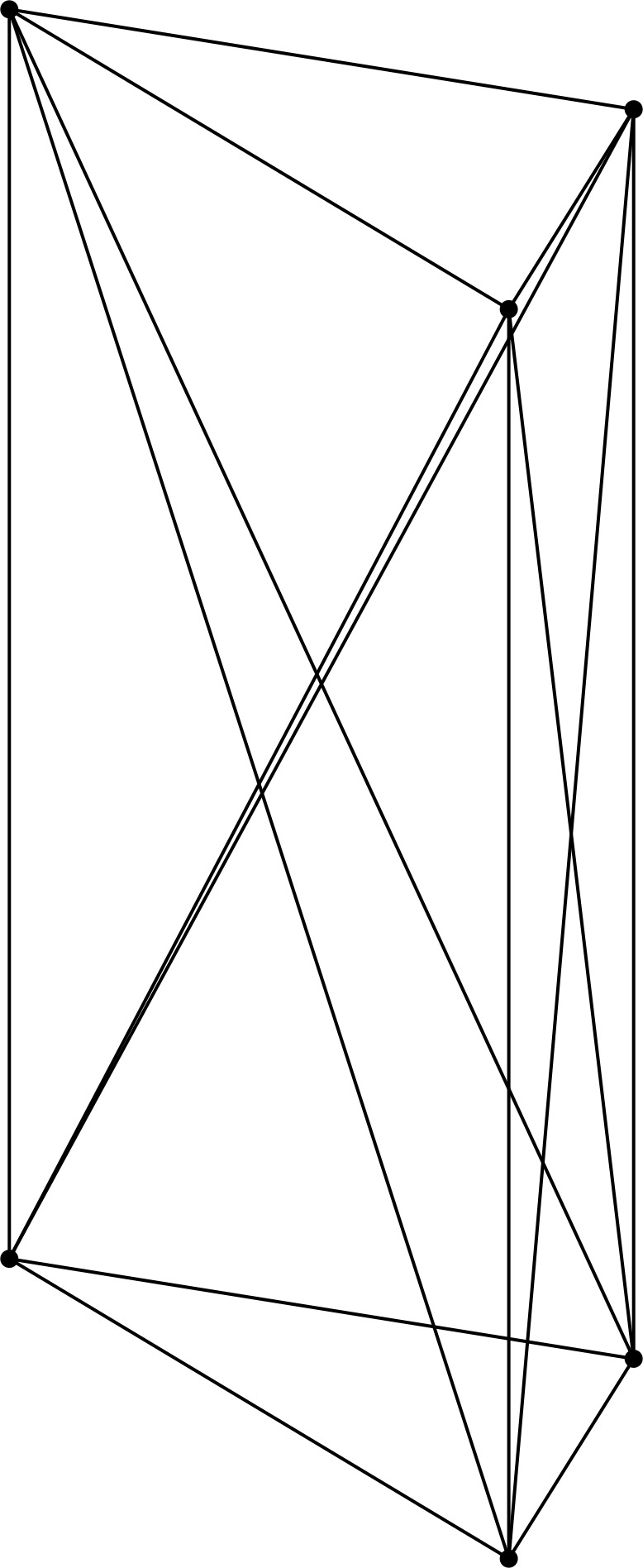}} &
	$\leftarrow$ &
	\includegraphics[height=4.8cm]{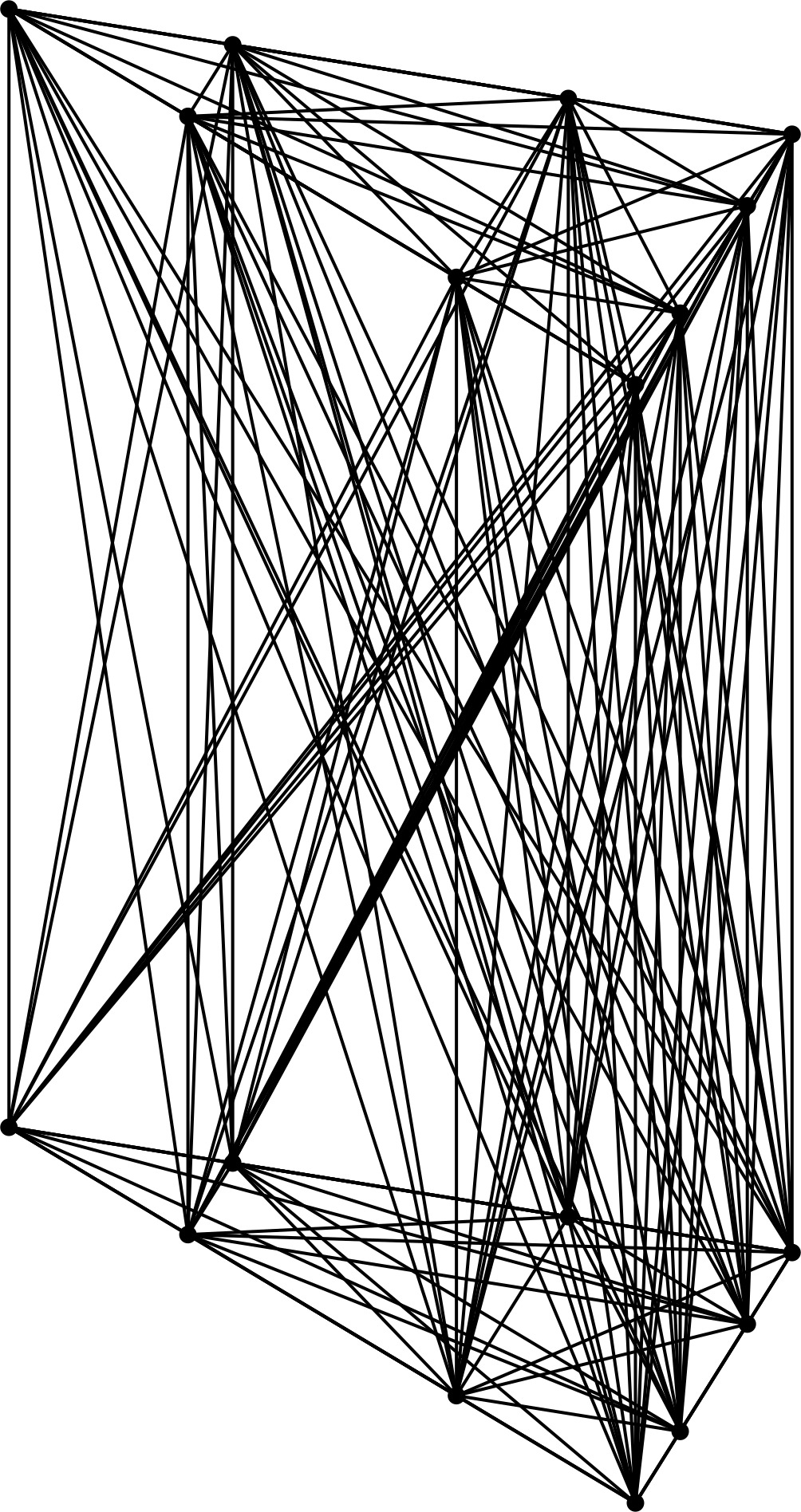} &
	$\leftarrow$ &
	$\cdots$
\end{longtable}

{\footnotesize {\bf Figure.} The sequence $K \leftarrow \N{}^1_3(K) \leftarrow \N{}^1_3(\N{}^1_3(K)) \leftarrow \cdots$ for a single-edge starting graph $K$. Here $\kappa = 3$; to get a N\"obeling manifold take infinite $\kappa$.}
\end{center}

\begin{lemma}
  The map $\pi \colon \nuk \to K$ is a quasi-simplicial map satisfying \qts.
\end{lemma}
\begin{proof}
  By the definition, $\pi$ is constructed to be a simplicial map into $\beta K$, hence it is quasi-simplicial into $K$. It satisfies \qts\ by Lemma~\ref{lem:qts}.
\end{proof}

\begin{construction}
  Let $K$ be a simplicial complex.
  Let $\kappa$ be a cardinal number.
  Let $n$ be a natural number.
  Let $K_0 = K$.
  Let $K_{i+1} = \N{}^n_\kappa(K_i)$ and let $p_{i+1}$ be the canonical map from $\N{}^n_\kappa(K_i)$ to $\beta K_i$.
  Let
  \[
    \nuk = \lim_{\longleftarrow} \left(  K_0 \xleftarrow{p_1} K_1 \xleftarrow{p_2} K_2 \xleftarrow{p_3} \cdots \right)
  \]
  We let $\pi_K \colon \nuk \to K$ denote the long projection.
\end{construction}

\begin{lemma}\label{lem:nuk is polish}
  $\nuk$ is an $n$-dimensional metric space of weight $\kappa$.
\end{lemma}
\begin{proof}
  We have $\nuk = \nu^n_\kappa(K^{(n)})$ (where $K^{(n)}$ denotes the $n$-dimensional skeleton of $K$) so without a loss of generality we may assume that $\dim K \leq n$.
  Then $\nuk$ is an inverse limit of a sequence of $n$-dimensional spaces, hence it is at most $n$-dimensional. It contains an $n$-dimensional simplex a subspace, hence it is $n$-dimensional.
  
  Since each $K_i$ is finite dimensional, it is complete. An inverse limit of complete spaces is complete.
\end{proof}

\begin{lemma}\label{lem:nuk is ane}
  $\nuk$ is an absolute neighborhood extensor in dimension $n$.
\end{lemma}
\begin{proof}
  Each $p_i$ is $n$-regular by Lemma~\ref{lem:qts fibers}. By Theorem~\ref{thm:ane limit of polyhedra}(1), $\nuk$ is $ANE(n)$.
\end{proof}

\begin{lemma}\label{lem:nuk is universal}
  $\nuk$ is strongly universal in a class of $n$-dimensional complete metric spaces of weight $\kappa$.
\end{lemma}
\begin{proof}
  This follows from Theorem~\ref{thm:awn universal}.
\end{proof}

\begin{lemma}\label{lem:nuk n-homotopy}
  The projection $\pi_K \colon \nuk \to K$ is an $n$-homotopy equivalence.
\end{lemma}
\begin{proof}
  Each $p_i$ is $n$-regular by Lemma~\ref{lem:qts fibers}. By Theorem~\ref{thm:ane limit of polyhedra}(2), $\nuk$ is $ANE(n)$.
\end{proof}

\begin{theorem}
  $\nuk$ is an $n$-dimensional abstract N\"obeling manifold of weight $\kappa$ and the projection $\pi \colon \nuk \to K$ is an $n$-homotopy equivalence.
\end{theorem}
\begin{proof}
  Apply Lemma~\ref{lem:nuk is polish}, \ref{lem:nuk is ane}, \ref{lem:nuk is universal}, \ref{lem:nuk n-homotopy}.
\end{proof}

\bibliographystyle{abbrv}
\bibliography{references2}

\end{document}